\documentclass[preprint]{elsarticle}

\usepackage{latexsym,amsfonts,amssymb,amsmath,amsthm}
\usepackage[T1]{fontenc}

\newcommand{\RR}{\ensuremath{\mathbb{R}}}

\newcommand{\p}{\ensuremath{\partial}}

\renewcommand{\theta}{\vartheta}

\newtheorem{Theorem}{Theorem}[section]

\newtheorem{definition}{Definition}

\numberwithin{equation}{section}

\begin{document}
\begin{frontmatter}

\title{Attractivity and stability in the competitive systems of PDEs of Kolmogorov type}
\author{Joanna Balbus}
\ead{joanna.balbus@pwr.wroc.pl}

\date{}

\address{Institute of Mathematics and Computer Science,
Wroc{\l}aw University of Technology, Wybrze\.ze Wyspia\'nskiego 27,
PL-50-370 Wroc{\l}aw, Poland}

\begin{abstract}
In this paper we consider $N$-species nonautonomous competitive
systems of partial differential equations of Kolmogorov type. Under
the Neumann boundary conditions we give a sufficient condition for
the system to be uniformly stable and globally attractive.
\end{abstract}

\end{frontmatter}

\section{Introduction}
In this paper we consider Kolmogorov systems of reaction-diffusion
parabolic partial differential equations (PDEs)
\begin{equation*}
\frac{\p u_i}{\p t}= \Delta u_i + f_i(t,x,u_1,\dots,u_N)u_i \quad t>0,\ x \in \Omega,\ 1\le i \le N.
\end{equation*}
From the biological viewpoint this problem represents a model of
population growth where $u_i(t,x)$ is the density of the $i$-th
species at time $t$ and spatial location $x \in \bar{\Omega}$,
$\Omega \subset R^n$ is a bounded habitat, $f_i(t,x, u_1, \dots,
u_N)$ is the local per capita growth rate of the $i$-th species. The
system is usually considered with appropriate boundary conditions,
either Dirichlet or Neumann, or Robin boundary conditions. In this
paper we investigate the system under the Neumann boundary
conditions. We deal with competitive systems. It means that the
derivatives $\frac{\p f_i}{\p u_j}$, $1\le i,j \le N$ are
nonpositive.

One of the major problems in models of population growth are
permanence and stability. The authors of many papers focus on finding
necessary or sufficient conditions for permanence or stability in
models which are described by ordinary or partial differential
equations. Permanence means that any positive solution of the model
has all its components, for large $t$, bounded away from zero and the
lower bound is independent of the solution.  Global attractivity
states that the difference of any two positive solutions tends to
zero as $t$ goes to infinity. In \cite{Ahmad} S. Ahmad and A.C.
Lazer, using the upper and lower averages of a function, found
sufficient conditions  for permanence and global attractivity of the
competitive $N$ species Lotka-Volterra system of ODEs. In \cite{moja}
we extended their result for the $N$ species competitive Kolmogorov
system of ODEs. In terms of the upper and lower averages of a
function we found sufficient conditions for such a system to be
permanent and globally attractive.

In \cite{Langa} the authors considered two species Lotka -- Volterra
system of PDEs. They investigated behavior of this system not only as
$t$ goes to infinity, but in the ''pullback'' sence. It means that we
are starting with a fixed initial condition further and further back
in time. They introduced a notion ''pullback permanence'' which has
an interesting biological interpretation. When two species have
already been competing for a long time then neither species will have
died out. Their densities are bounded below in a uniform way, no
matter how long this population has been running.

In \cite{M-B} we investigate the $N$ species competitive Kolmogorov
system of PDEs. Using the methods of supersolutions and subsolutions
for parabolic PDEs we found sufficient conditions for permanence in
such systems.

To establish the global attractivity many authors use appropriate
Lyapunov function in the case of ODE or Lyapunov functional in the
case of PDE (see for example \cite{Ahmad}, \cite{moja}, \cite{RR1},
\cite{Sze-Bi}, \cite{Tineo}). In  \cite{Gopalsamy-almost-periodic,
Gopalsamy-periodic} K. Gopalsamy dealt with, respectively,  periodic
and almost periodic solution of a Lotka - Volterra system of ODEs. A.
Tineo \cite{Tineo} considered the nonautonomous  Kolmogorov system of
ODEs
\begin{equation}
\label{Kolmogorov}
\tag{K}
u_i'=f_i(t,u)u_i \quad 1 \le i \le N,
\end{equation}
where $f:\RR \times \RR_+^N \to \RR^N$ is a continuous function and
$\RR_+^N$ is the nonnegative orthant $\RR_+^N=\{u\in \RR^N:\ u_i\geq
0, 1 \le i \le N\}$. He found a sufficient condition for the system
\eqref{Kolmogorov} to be globally attractive. In \cite{moja} we
considered system \eqref {Kolmogorov} with slightly weaker
assumptions for the function $f$ than those $f$ in \cite{Tineo}. In
\cite{Sze-Bi} Sze-Bi Hsu examined Lyapunov functions (or functionals)
for various ecological models which take form of ODE systems (or
reaction-diffusion PDE systems). First the author constructed
Lyapunov functions for the predator prey system and then constructed
Lyapunov functionals for the corresponding reaction-diffusion PDE
system. Sze-Bi Hsu considered the models in which the reaction term
is independent on $t$ and $x$.

In \cite{Riviero} A. Tineo and J. Riviero dealt with the two species
nonautonomous competitive Lotka-Volterra system with diffusion. They
resigned from the Lyapunov function (functional) to establish of
attractivity of this system. They used a powerful tool which is the
iteration monotone methods.

Usually the same conditions which guarantee permanence in the ODE
model of population growth ensure that the model is globally
attractive ( see for example \cite{Ahmad}, \cite{moja}). But when we
try to prove that a system of Kolmogorov type of PDEs is globally
attractive then some difficulties occur. Firstly, we can prove that
the system is globally attractive only under the Neumann boundary
conditions. Secondly, besides the conditions which guarantee the
permanence we need additional inequalities which contain the
parameters which occur in the definition of permanence. Even in the
case of two species problem it is a difficult to obtain sufficient
conditions which guarantee both permanence and global attractivity of
Kolmogorov system of PDEs. One can try to introduce a partial
ordering. The idea is that we can compare solutions of our system of
PDE with the functions which are independent on $x$ in a way that
when we apply partial ordering our solutions of PDE are between the
functions independent on $x$. A difficulty occurs when we let the
parameter which is a spatial location to be arbitrary.

\bigskip
In this paper we find sufficient conditions for $N$ -- species system
of PDEs of Kolmogorov type under the Neumann boundary conditions  to
be globally attractive. Conditions which we obtained are stronger
than those guaranteeing perrmanence. Assumptions for a function $f$
in this paper is the same as in \cite{M-B}.

This paper is organized as follows.

In Section 2 we present basic assumptions and definitions.

In Section 3 we formulate the main theorem of this paper. By Lyapunov
functional some sufficient conditions are obtained to guarantee
attractivity and uniform stability Kolmogorov system of PDEs.

In Section 4 we give some remarks.

\section{Preliminaries}
We consider a nonautonomous system of PDEs of Kolmogorov type
\begin{equation*}
\label{R}
\tag{R}
\begin{cases}
\displaystyle
\frac{\p u_i}{\p t} = \Delta u_i + f_i(t,x,u_1, \dots, u_N) u_{i},
& \quad t > 0, \  x \in \Omega,  \  i = 1,\dots, N \\[2ex]
\mathcal{B}u_i = 0, & \quad t > 0, \  x \in \partial\Omega, \  i
= 1,\dots, N,
\end{cases}
\end{equation*}
where $\Omega \subset \RR^n$ is a bounded domain with a sufficiently
smooth
 boundary $\partial\Omega$,
 $\Delta$ is the Laplace operator on $\Omega$,
\begin{equation*}
\Delta = \frac{\p^2}{\p x_{1}^{2}} + \dots + \frac{\p^2}{\p
x_{n}^{2}},
\end{equation*}
  and $\mathcal{B}$ is the boundary
operator of the Neumann type
\begin{equation*}
\mathcal{B}u_i = \frac{\partial u_i}{\partial \nu} \quad \text{on
}\partial\Omega,
\end{equation*}
where $\nu$ denotes the unit normal vector on $\p\Omega$ pointing
out~of $\Omega$.

Now we make the following assumptions

(A1) (see \cite{M-B}) {\em $f_{i} \colon [0, \infty) \times
\bar{\Omega} \times [0,\infty)^{N} \to \RR$ \textup{(}$1 \le i \le
N$\textup{)}, as~well as their first derivatives $\p f_{i}/\p t$
\textup{(}$1 \le i \le N$\textup{)}, $\p f_i/\p u_j$ \textup{(}$1 \le
i,j \le N$\textup{)}, and $\p f_i/\p x_k$ \textup{(}$1 \le i \le N$,
$1 \le k \le n$\textup{)}, are continuous.}

\medskip
(A1) guarantees that for any  regular initial function $u_0(x)=
(u_{01}, \dots, u_{0N})$, $x \in \Omega$ where $u_{0i}\geq 0$ ($1\le
i \le N$, $x \in \Omega$) there exists a unique maximally defined
solution $u(t,x)=(u_1(t,x), \dots, u_N(t,x))$ of \eqref{R}, $(t,x)\in
[0, \tau_{\max})\times \bar{\Omega}$, where $\tau_{\max}>0$,
satisfying the initial condition $u(0,x)=u_0(x)$. Moreover, the
solution of system \eqref{R} is classical.

\medskip
(A2)(see \cite{M-B}) {\em Functions $[\, [0,\infty) \times \bar\Omega
\ni (t,x) \mapsto f_{i}(t,x,0, \dots, 0) \in \RR \,]$, $1 \le i \le
N$, are bounded.}

\begin{definition}
A solution $u(t,x)=(u_1(t,x), \dots, u_N(t,x))$ of system \eqref{R}
is \emph{positive} if $u_i(t,x)>0$ for all $i=1,\dots,N$, $x\in
\Omega$, $t>0$.
\end{definition}

\begin{definition}
System \eqref{R} is \emph{permanent} if there exist positive
constants $\underline{\delta}$, $\overline{\delta}$ such that for any
positive solution $u(t,x)=(u_1(t,x), \dots, u_N(t,x))$ there exists
$T=T(u)\geq0$ with the property
\begin{displaymath}
\underline{\delta} \le u_i(t,x) \le \overline{\delta}
\end{displaymath}
for all $t \geq T$, $x \in \bar{\Omega}$, $i=1, \dots, N$.
\end{definition}
\begin{definition}
System \eqref{R} is \emph{globally attractive} if any two positive
solutions $u(t,x)=(u_1(t,x), \dots, u_N(t,x))$ and $v(t,x)=(v_1(t,x),
\dots, v_N(t,x))$ of system \eqref{R} satisfy
\begin{equation*}
\lim_{t \to \infty} (u_i(t,x) - v_i(t,x))=0
\end{equation*}
for $1 \le i\le N$, uniformly in $x \in \bar{\Omega}$.
\end{definition}
\begin{definition}
System \eqref{R} is \emph{uniformly stable} on $[t_1, \infty)$, $t_1>
0$ if for every $\varepsilon>0$ there is a $\delta>0$ so that for any
two positive solutions $u(t,x)=(u_1(t,x),\dots, u_N(t,x))$  and
$v(t,x)=(v_1(t,x),\dots, v_N(t,x))$  if $t_2\in [t_1, \infty)$ and
$\left\| u(t_2,x)-v(t_2,x)\right\| < \delta$ then $\left\| u(t,x) -
v(t,x) \right\| < \varepsilon$ for $t\geq t_2$.
\end{definition}

\bigskip
The next assumption is

(A3)(see \cite{M-B}) {\em There exist $\underline{b}_{ii} > 0$ such
that $\frac{\partial f_i}{\partial u_i}(t,x,u) \le
-\underline{b}_{ii}$ for all $t \geq 0$, $x \in \bar{\Omega}$, $u \in
[0,\infty)^{N}$, $1 \le i \le N$.}

\medskip
Define
\begin{equation*}
B(\varepsilon) := \left[0,\frac{\bar{a}_{1}}{\underline{b}_{11}} +
\varepsilon\right] \times\dots \times
\left[0,\frac{\bar{a}_{N}}{\underline{b}_{NN}} + \varepsilon\right],
\quad \varepsilon \ge 0.
\end{equation*}

\bigskip
(A4)(see \cite{M-B})  {\em $\partial f_i/\partial u_{j}$, $1 \le i, j
\le N$, are bounded and uniformly continuous on each set $[0,\infty)
\times \bar{\Omega} \times B$.}

\medskip
For $1 \le i, j \le N$ define
\begin{equation*}
\overline{b}_{ij}(\epsilon):=\sup \biggl\{\, -\frac{\partial f_{i}}{\partial u_{j}}(t,x,u): t \ge 0,\ x
\in \bar{\Omega}, \ u \in B(\varepsilon) \, \biggr\}.
\end{equation*}
and $\overline{b}_{ij}(0):=\overline{b}_{ij}$.

\medskip
Assumptions (A3) and (A4) imply that $\overline{b}_{ij}(\epsilon) \ge
0$, $1 \le i, j, \le N$, and $\overline{b}_{ii}(\epsilon) > 0$, $1
\le i \le N$, whereas it follows from (A5) that
$\overline{b}_{ij}(\epsilon) < \infty$, and $\lim_{\epsilon \to
0^{+}} \overline{b}_{ij}(\epsilon) = \overline{b}_{ij}$, for $1 \le
i, j \le N$.

\bigskip
Recall that in \cite{M-B} we found average conditions for system
\eqref{R} to be permanent. The average conditions for system
\eqref{R} under the Neumann boundary conditions are
\begin{equation}
\label{AC}
\tag{AC}
m[f_i]>\sum_{\substack{j=1 \\ j\neq i}}^N \frac{\bar{b}_{ij}M[f_j]}{\underline{b}_{jj}},
\end{equation}
where
\begin{displaymath}
m[f_i]:= \liminf_{t-s\to \infty}\frac{1}{t-s}\int_s^t \min_{x\in
\bar{\Omega}}f_i(\tau,x,0,\dots,0)d\tau,
\end{displaymath}
and
\begin{displaymath}
M[f_i]:= \limsup_{t-s\to \infty}\frac{1}{t-s}\int_s^t \max_{x\in
\bar{\Omega}}f_i(\tau,x,0,\dots,0)d\tau.
\end{displaymath}

\section{Main Theorem}
In this section we state the main theorem of this paper and with the
help of a Lyapunov functional we prove attractivity and stability of
system \eqref{R}.
\begin{Theorem}
\label{Lyapunov}
Assume (A1) -- (A5) and (AC). Let $\underline{\delta},
\overline{\delta} \geq 0$ be such that for each $1 \le i\le N$
\begin{equation}
\label{2.1}
\tag{2.1}
\underline{\delta}\, \underline{b}_{ii} >
\sum_{\substack{j=1
\\j\neq i}}^N \overline{\delta}\, \overline{b}_{ji}.
\end{equation}
Suppose that $u(t,x) = (u_1(t,x), \dots, u_N(t,x))$ and $v(t,x) =
(v_1(t,x),\dots,
 v_N(t,x))$ be two positive solutions of system \eqref{R} such that $\underline{\delta}\le u_i(t,x),v_i(t,x)\le \overline{\delta}$
 for sufficiently large $t$ and all $x \in \bar{\Omega}$.
 Then there exist positive constants $Z$, $\gamma$ and $\tilde{T} \ge 0$ such that for
$i = 1,\dots,N$ and $t \geq \tilde{T}$ one has
\begin{displaymath}
\sum_{i=1}^N \sup_{x \in \bar{\Omega}} \lvert u_i(t,x) -
v_i(t,x)\rvert \le Z \sum_{i=1}^N \sup_{x \in \bar{\Omega}} \lvert
u_i(\tilde{T},x) - v_i(\tilde{T},x)\rvert \cdot \exp{(-\gamma
(t-\tilde{T}))},
\end{displaymath}
In particular system \eqref{R} is globally attractive and uniformly
stable.
\end{Theorem}
\begin{proof}
Fix two positive solutions $u(t,x)=(u_1(t,x),\dots,u_N(t,x))$ and
$v(t,x)=(v_1(t,x),\dots,v_N(t,x))$ of system \eqref{R}. By assumption
(A5) we can take $\varepsilon>0$ such that
\begin{equation*}
\underline{\delta}_i\, \underline{b}_{ii} >
\sum_{\substack{j=1
\\j\neq i}}^N  \overline{\delta}_{j}\, \bar{b}_{ij}(\varepsilon) \quad \text{for all} \quad i=1,\dots,N.
\end{equation*}
Let $t_0>0$ be such that $u(t,x), v(t,x)\in B(\varepsilon)$ and
$\underline{\delta} \le u_i(t,x),v_i(t,x) \le \overline{\delta}$ for
all $t\geq t_0$, $x \in \bar{\Omega}$, $1\le i \le N$. By a result in
matrix theory (see  \cite[Section 5]{RR1}) we know that there exist
$\alpha_1, \dots, \alpha_N>0$ such that
\begin{equation}
\label{dominacja-1}
\alpha_i\underline{\delta}\, \underline{b}_{ii} >
\sum_{\substack{j=1
\\j\neq i}}^N \alpha_j \overline{\delta} \bar{b}_{ji}(\varepsilon), \quad 1\le i \le N.
\end{equation}
Denote
\begin{equation}
\label{2.2}
\Theta(t) = \sum_{i=1}^{N} \alpha_{i} \Theta_{i}(t), \; \text{where }
\Theta_{i}(t):= \sup\limits_{x \in \bar{\Omega}} \left\lvert \ln
\frac{u_i(t,x)}{v_i(t,x)} \right\rvert.
\end{equation}
Fix $1 \le i \le N$. We prove that there exists $\tilde{T}\geq t_0$
such that
\begin{equation}
\label{Lapunow-nierownosc}
D^{+}\Theta_i(t) \le -\underline{\delta} \, \underline{b}_{ii}
\Theta_i(t) + \overline{\delta} \sum_{\substack{j=1\\ j\neq i}}^N
\overline{b}_{ij}(\varepsilon) \Theta_j(t), \qquad t \ge \tilde{T}.
\end{equation}
where $D^{+}\Theta$ denotes the upper derivative of $\Theta$.

\medskip
We take for $\tilde{T}$ any number $\geq t_0$. For $t\geq \tilde{T}$
denote $\bar{\Omega}^{(i)}(t)$ to be the set of the points
$x^{(i)}\in \bar{\Omega}$ such that the supremum of the function
$\left\lvert \ln\frac{u_i(t,x)}{v_i(t,x)}\right\rvert$ is realized.
Then
$\bar{\Omega}^{(i)}(t)=\bar{\Omega}^{(i)}_{+}(t)\cup\bar{\Omega}^{(i)}_{-}(t)$,
$\bar{\Omega}^{(i)}_{+}(t)\cap\bar{\Omega}^{(i)}_{-}(t)= \emptyset$,
where
\begin{displaymath}
\bar{\Omega}^{(i)}_{+}(t) = \{x^{(i)} \in \bar{\Omega}^{(i)}(t) :
u_i(t,x^{(i)})>v_i(t,x^{(i)}) \}
\end{displaymath}
and
\begin{displaymath}
\bar{\Omega}^{(i)}_{-}(t) = \{x^{(i)} \in \bar{\Omega}^{(i)}(t) :
u_i(t,x^{(i)})<v_i(t,x^{(i)}) \}
\end{displaymath}
Therefore by \eqref{2.2}
\begin{equation}
\label{Lapunow-1}
\begin{aligned}
D^{+}\Theta_i(t) \le \max\biggl\{&\sup \Bigl\{\frac{\p}{\p t}
\ln\frac{u_i(t,x^{(i)})}{v_i(t,x^{(i)})} : x^{(i)} \in
\bar{\Omega}_+^{(i)}(t) \Bigr\},
\\
& \sup \Bigl\{\frac{\p}{\p t}  \ln\frac{v_i(t,x^{(i)})}{u_i(t,x^{(i)})}
: x^{(i)} \in \bar{\Omega}_{-}^{(i)}(t) \Bigr\} \biggr\}.
\end{aligned}
\end{equation}
Denote $\kappa_i:=\Theta_i(t)$. Now we consider four (not mutually
exclusive) cases.
\newline
\textit{Case one. $x^{(i)}\in\Omega\cap\bar{\Omega}^{(i)}_{+}(t)$.}
\newline
By the definition of $\kappa$ it follows that
\begin{equation}
\label{2.3}
u_i(t,x)\le e^{\varkappa_i}v_i(t,x) \quad \text{for all } x \in
\bar{\Omega} \quad \text{and} \quad u_i(t,x^{(i)}) =
e^{\varkappa_i}v_i(t,x^{(i)}),
\end{equation}
and
\begin{align*}
\frac{d}{dt}\left( \ln\frac{u_i(t,x^{(i)})}{v_i(t,x^{(i)})} \right) =
\frac{\frac{d}{dt}u_i{(t,x^{(i)})}}{u_i(t,x^{(i)})} -
\frac{\frac{d}{dt}v_i{(t,x^{(i)})}}{v_i(t,x^{(i)})}
\\
= \frac{\Delta u_i{(t,x^{(i)})}}{u_i(t,x^{(i)})} - \frac{\Delta
v_i{(t,x^{(i)})}}{v_i(t,x^{(i)})} + f_i(t,x^{(i)},u(t,x^{(i)}))-
f_i(t,x_i,v(t,x^{(i)})).
\end{align*}
By \eqref{2.3},
\begin{displaymath}
\frac{\Delta u_i{(t,x^{(i)})}}{u_i(t,x^{(i)})} - \frac{\Delta
v_i{(t,x^{(i)})}}{v_i(t,x^{(i)})}=
\frac{\Delta(u_i(t,x^{(i)})-e^{\varkappa_i}v_i(t,x^{(i)}))}{u_i(t,x^{(i)})}\le
0.
\end{displaymath}
Hence
\begin{displaymath}
\frac{d}{dt}\left( \ln\frac{u_i(t,x^{(i)})}{v_i(t,x^{(i)})}
\right)\le f_i(t,x^{(i)},u(t,x^{(i)}))- f_i(t,x^{(i)},v(t,x^{(i)})).
\end{displaymath}
\textit{Case two. $x^{(i)}\in\Omega\cap\bar{\Omega}^{(i)}_{-}(t)$.}

In a similar manner we prove that
\begin{equation*}
\frac{d}{dt}\left( \ln\frac{v_i(t,x^{(i)})}{u_i(t,x^{(i)})}
\right)\le f_i(t,x^{(i)},v(t,x^{(i)}))- f_i(t,x^{(i)},u(t,x^{(i)})).
\end{equation*}

\textit{Case three. $x^{(i)}\in
\partial\Omega\cap\bar{\Omega}^{(i)}_{+}(t)$}

This implies that $u_i(t,x) \le e^{\varkappa_i}v_i(t,x^{(i)})$ for
all $x \in \Omega$ and $u_i(t,x^{(i)}) =
e^{\varkappa_i}v_i(t,x^{(i)})$. Denote $g(x):= u_i(t,x) -
e^{\varkappa_i}v_i(t,x)$. The restriction $g|_{\p\Omega}$ of  $g$ to
$\partial\Omega$ has largest value at the point $x^{(i)}$,
consequently $\nabla(g|_{\p\Omega})(x^{(i)}) = 0$. By the definition
of $g$ it follows that $\frac{\partial g}{\partial \nu}=
\frac{\partial u_i}{\partial \nu} - e^{\varkappa_i}\frac{\partial
v_i}{\partial \nu} = 0$.  Hence $\nabla g(x^{(i)}) = 0$.

As $g$ attains its maximum at $x^{(i)}$, this implies that $\Delta
g(x^{(i)}) \le 0$.  Therefore
\begin{align*}
\frac{d}{dt}\left( \ln\frac{u_i(t,x^{(i)})}{v_i(t,x^{(i)})} \right) =
\frac{\frac{d}{dt}u_i{(t,x^{(i)})}}{u_i(t,x^{(i)})} -
\frac{\frac{d}{dt}v_i{(t,x^{(i)})}}{v_i(t,x^{(i)})}
\\
= \frac{\Delta u_i{(t,x^{(i)})}}{u_i(t,x^{(i)})} - \frac{\Delta
v_i{(t,x^{(i)})}}{v_i(t,x^{(i)})} +f_i(t,x^{(i)},u(t,x^{(i)}))-
f_i(t,x^{(i)},v(t,x^{(i)}))
\\
\le f_i(t,x^{(i)},u(t,x^{(i)}))- f_i(t,x^{(i)},v(t,x^{(i)})).
\end{align*}

\textit{Case four. $x^{(i)}\in\Omega\cup\bar{\Omega}^{(i)}_{-}(t)$}

In a similar manner we prove that
\begin{equation*}
\frac{d}{dt}\left( \ln\frac{v_i(t,x^{(i)})}{u_i(t,x^{(i)})} \right)
\le f_i(t,x^{(i)},v(t,x^{(i)}))- f_i(t,x^{(i)},u(t,x^{(i)})).
\end{equation*}

The cases 1 - 4 with \eqref{Lapunow-1} give
\begin{align*}
D^+\Theta_{i}(t) \le \sum_{i=1}^N \max\{\sup\{f_i(t,x^{(i)},
u(t,x^{(i)}))- f_i(t,x^{(i)},v(t,x^{(i)})) : x^{(i)} \in
\bar{\Omega}_+^{(i)}\},
\\
\sup\{f_i(t,x^{(i)}, u(t,x^{(i)}))- f_i(t,x^{(i)},v(t,x^{(i)})) :
x^{(i)} \in \bar{\Omega}_-^{(i)}\}  \}.
\end{align*}
By assumptions (A3), (A4) and (A5), it follows that
\begin{equation*}
\begin{split}
f_i(t,x^{(i)},u_1(t,x^{(i)}), \dots, u_N(t,x^{(i)})) -
f_i(t,v_1(t,x^{(i)})), \dots, v_N(t,x^{(i)})) \\
\le -\underline{b}_{ii} (u_i(t,x^{(i)}) - v_i(t,x^{(i)})) +
\sum_{\substack{j=1 \\ j\neq i}}^N \overline{b}_{ij}(\varepsilon)
\lvert u_j(t,x^{(i)}) - v_j(t,x^{(i)}) \rvert
\end{split}
\end{equation*}
for each $x^{(i)} \in \bar{\Omega}_+^{(i)}(t)$.

\smallskip
By assumptions (A3), (A4) and (A5), it follows that
\begin{equation*}
\begin{split}
f_i(t,x^{(i)}),u_1(t,x^{(i)}),\dots, u_N(t,x^{(i)}) -
f_i(t,v_1(t,x^{(i)})),\dots, v_N(t,x^{(i)}))  \\
\le -\underline{b}_{ii} (v_i(t,x^{(i)}) - u_i(t,x^{(i)})) +
\sum_{\substack{j=1 \\ j\neq i}}^N \overline{b}_{ij}(\varepsilon)
\lvert v_j(t,x^{(i)}) - u_j(t,x^{(i)}) \rvert
\end{split}
\end{equation*}
for each $x^{(i)} \in \bar{\Omega}_{-}^{(i)}(t)$.

\smallskip
Hence
\begin{equation*}
\begin{split}
D^{+} \Theta_{i}(t) \le \max\{ \sup_{x^{(i)} \in
\bar{\Omega}_+^{(i)}} \{ -\underline{b}_{ii} (u_i(t,x^{(i)}) -
v_i(t,x^{(i)})) + \sum_{\substack{j=1 \\ j\neq i}}^N
\overline{b}_{ij}(\varepsilon) \lvert u_j(t,x^{(i)}) - v_j(t,x^{(i)})
\rvert \},
\\
\sup_{x^{(i)} \in \bar{\Omega}_-^{(i)}} \{ -\underline{b}_{ii}
(v_i(t,x^{(i)}) - u_i(t,x^{(i)})) + \sum_{\substack{j=1 \\ j\neq
i}}^N \overline{b}_{ij}(\varepsilon) \lvert v_j(t,x^{(i)}) -
u_j(t,x^{(i)}) \rvert \} \}
\end{split}
\end{equation*}

Since $\delta_i \le u_i(t,x), v_i(t,x) \le R_i$ for all $i=1,\dots,
N$ and for sufficiently large $t$, using the mean value theorem we
have that
\begin{equation}
\label{rrc2.6}
\frac{1}{\overline{\delta}} \lvert u_j(t,x) - v_j(t,x) \rvert \le
\Bigl\lvert \ln{\frac{u_j(t,x)}{v_j(t,x)}} \Bigr\rvert \le
\frac{1}{\underline{\delta}} \lvert u_j(t,x) - v_j(t,x) \rvert
\end{equation}
for all $t \ge \tilde{T}$, $x \in \bar{\Omega}$, $1 \le j \le N$.

Therefore
\begin{equation*}
\begin{split}
D^+\Theta_{i}(t) \le \max\Bigl\{\sup\{-\underline{\delta} \,
\underline{b}_{ii}
\Bigl\lvert\ln\frac{u_i(t,x^{(i)})}{v_i(t,x^{(i)})} \Bigr\rvert +
\overline{\delta} \sum_{\substack{j=1\\j\neq i}}^N
\overline{b}_{ij}(\varepsilon)
\Bigl\lvert\ln\frac{u_j(t,x^{(i)})}{v_j(t,x^{(i)})} \Bigr\rvert :
x^{(i)} \in \bar{\Omega}_+^{(i)}\}, \\ \sup\{-\underline{\delta} \,
\underline{b}_{ii}
\Bigl\lvert\ln\frac{u_i(t,x^{(i)})}{v_i(t,x^{(i)})} \Bigr\rvert +
\overline{\delta} \sum_{\substack{j=1\\j\neq i}}^N
\overline{b}_{ij}(\varepsilon)
\Bigl\lvert\ln\frac{u_j(t,x^{(i)})}{v_j(t,x^{(i)})} \Bigr\rvert :
x^{(i)} \in \bar{\Omega}_-^{(i)}\} \Bigr\}
\\
\le \Bigl( -\underline{\delta} \, \underline{b}_{ii} \sup\limits_{x \in
\bar{\Omega}} \left\lvert \ln \frac{u_i(t,x)}{v_i(t,x)}\right\rvert +
\overline{\delta} \sum_{\substack{j=1\\ j\neq i}}^N
\overline{b}_{ij}(\varepsilon) \sup\limits_{x \in \bar{\Omega}} \left\lvert
\ln \frac{u_j(t,x)}{v_j(t,x)}\right\rvert \Bigr), \quad t\geq \tilde{T}.
\end{split}
\end{equation*}
Hence we proved \eqref{Lapunow-nierownosc}.

By \eqref{dominacja-1}, it follows that
\begin{equation}
\label{nierownosc-1}
\begin{aligned}
& \qquad \sum_{i=1}^N \alpha_{i} \biggl( -\underline{\delta} \,
\underline{b}_{ii} \Theta_i(t) + \overline{\delta} \sum_{\substack{j=1\\
j\neq i}}^N \overline{b}_{ij}(\varepsilon) \Theta_j(t) \biggr)
\\
& = -\sum_{i=1}^N  \alpha_{i} \underline{\delta} \,
\underline{b}_{ii}
\Theta_i(t) + \overline{\delta} \sum_{j=1}^N \ \sum_{\substack{i=1\\
i\neq j}}^N \alpha_{j} \overline{b}_{ji}(\varepsilon) \Theta_i(t)
\\
& = -\sum_{i=1}^N \Bigl( \alpha_{i} \underline{\delta} \,
\underline{b}_{ii} - \sum_{\substack{j=1\\ j\neq i}}^N \alpha_{j}
\overline{\delta} \, \overline{b}_{ji}(\varepsilon) \Bigr)
\Theta_i(t)
\\
& \le - \epsilon \sum_{i=1}^N \Theta_i(t), \quad t\geq \tilde{T},
\end{aligned}
\end{equation}
where
\begin{equation*}
\epsilon = \min\limits_{1 \le i \le N} \biggl\{ \alpha_{i} \underline{\delta} \,
\underline{b}_{ii} - \sum_{\substack{j=1\\ j\neq i}}^N \alpha_{j}
\overline{\delta} \, \overline{b}_{ji}(\varepsilon) \biggr\}.
\end{equation*}
Hence by \eqref{Lapunow-nierownosc} and \eqref{nierownosc-1},
\begin{equation}
\label{stability}
D^{+}\Theta(t) \le \sum\limits_{i=1}^{N} \alpha_{i} D^{+} \Theta_i(t)
\le - \epsilon \sum_{i=1}^N \Theta_i(t) \le
-\frac{\epsilon}{\alpha^{*}} \Theta(t), \quad t\geq \tilde{T}
\end{equation}
where $\alpha^{*} = \max\{\alpha_i : 1 \le i \le N \}$. Therefore
\begin{equation}
\label{wykladniczy}
\Theta(t) \le \Theta(\tilde{T}) \exp{\Bigl( -\frac{\epsilon}{\alpha^{*}}(t - \tilde{T}) \Bigr)}
\end{equation}
for all $t \ge \tilde{T}$.


By \eqref{rrc2.6} it follows that
\begin{equation}
\label{mean-value1}
\Theta(t) \ge \alpha_{*} \sum\limits_{i=1}^{N} \Theta_{i}(t) \ge
\frac{\alpha_{*}}{\overline{\delta}} \sum_{i=1}^N \sup_{x \in
\bar{\Omega}} \lvert u_i(t,x) - v_i(t,x) \rvert, \quad t \ge \tilde{T},
\end{equation}
and
\begin{equation}
\label{mean-value2}
\Theta(\tilde{T}) \le \alpha^{*} \sum\limits_{i=1}^{N} \Theta_{i}(t) \le
\frac{\alpha^{*}}{\underline{\delta}}\sum_{i=1}^N \sup_{x \in
\bar{\Omega}} \lvert u_i(\tilde{T},x) - v_i(\tilde{T},x) \rvert, \quad t\geq \tilde{T}
\end{equation}
where $\alpha_{*} = \min\{\alpha_i : 1 \le i \le N \}$.

From \eqref{wykladniczy}, \eqref{mean-value1} and \eqref{mean-value2}
we have that
\begin{displaymath}
\label{3.12}
\tag{3.12} \sum_{i=1}^N \sup_{x \in \bar{\Omega}} \lvert u_i(t,x) -
v_i(t,x)\rvert \le Z \sum_{i=1}^N \sup_{x \in \bar{\Omega}} \lvert
u_i(\tilde{T},x) - v_i(\tilde{T},x)\rvert \cdot \exp{(-\gamma
(t-\tilde{T}))}, \quad t\geq \tilde{T},
\end{displaymath}
where $Z = \frac{\overline{\delta} \alpha^*}{\underline{\delta}
\alpha_*}$ and $\gamma = \frac{\epsilon}{\alpha^*}$ ($Z$ and $\gamma$
we can take independently on the solutions $u(t,x)$ and $v(t,x)$).

By \eqref{3.12} it follows that the system \eqref{R} is globally
attractive and uniformly stable.

\bigskip

\end{proof}

\section{Remarks}

\emph{Remark 1.} Note that conditions  \eqref{dominacja-1} contain
$\underline{\delta}$ and $\overline{\delta}$ which appear in
definition of permanence of system \eqref{R}. In \cite{M-B} we showed
that if we change conditions \eqref{AC} to slightly stronger that is
\begin{displaymath}
\label{AC'}
\tag{AC'} m[f_i]> \sum_{\substack{j=1 \\ j\neq i}}^N
\frac{\overline{b}_{ij}\overline{a}_j}{\underline{b}_{jj}}
\end{displaymath}
then we can estimate $\underline{\delta}$ and $\overline{\delta}$ in
terms of parameters of system \eqref{R}. By dissipativity of system
\eqref{R} it follows that we can take  $\max_{1\le i \le
N}\{\frac{\overline{a}_i}{\underline{b}_{ii}}\}$ as
$\overline{\delta}$. Concerning the lower bound (see \cite[Theorem
4.1]{M-B}) as $\underline{\delta}$ we can take
\begin{equation*}
\min_{1\le i\le N} \{\frac{1}{\overline{b}_{ii}}(\underline{a}_i- \sum_{\substack{j=1\\ j \neq i}}^N \frac{\overline{b}_{ij} \overline{a}_{j}}{\underline{b}_{jj}})\}
\end{equation*}
when
\begin{equation*}
\underline{a}_i>\sum_{\substack{j=1\\j\neq i}}^N \frac{\overline{b}_{ij} \overline{a}_{j}}{\underline{b}_{ij}}
\end{equation*}
for all $1 \le i \le N$.

If we define $\underline{\delta}$ and $\overline{\delta}$ as in the
above condition \eqref{2.1}, take a form
\begin{displaymath}
\alpha_i m[f_i]\geq \sum_{\substack{j=1\\j\neq i}}^N\left(
\alpha_i\frac{\overline{b}_{ij}M[f_j]}{\underline{b}_{jj}} +
\alpha_j\frac{\overline{b}_{ji}M[f_i]}{\underline{b}_{ii}}   \right).
\end{displaymath}
\emph{Remark 2.}
\newline
\begin{definition}
System \eqref{R} is \emph{uniformly asymptotically stable} on $[t_1,
\infty)$ if it is uniformly stable and if there is a $r>0$ so that
for every $\varepsilon>0$ there is a $T(\varepsilon)>0$ so that for
any two positive solutions $u(t,x)=(u_1(t,x), \dots, u_N(t,x))$ and
$v(t,x)=(v_1(t,x), \dots, v_N(t,x))$ if $t_2\geq t_1$ and
$\left\|u(t_2,x)-v(t_2,x) \right\| <r$ then $\left\|u(t,x)-v(t,x)
\right\| <\varepsilon$ for $t\geq t_2 + T(\varepsilon)$.
\end{definition}
Note that from attractivity and uniform stability  of system
\eqref{R} it follows that system \eqref{R} is uniformly
asymptotically stable.
\section{Acknowledgment}
The author thanks Professor Janusz Mierczy\'nski for helpful
discussions du\-ring writing of this manuscript.

\end{document}